\documentclass[11pt]{article}
\usepackage[english]{babel}
\usepackage{amssymb,amsmath,amsthm}
\newtheorem{theorem}{Theorem}[section]
\newtheorem{lemma}[theorem]{Lemma}

\newtheorem{corollary}[theorem]{Corollary}

\newtheorem{example}[theorem]{Example}

\newtheorem{conj}[theorem]{Conjecture}

{\theoremstyle{definition}}
{\theoremstyle{definition}\newtheorem{definition}[theorem]{Definition}}
{\theoremstyle{definition}}

\newtheorem*{thmgs}{Theorem GS}
\newtheorem*{corgs}{Corollary GS}
\newtheorem*{thmw}{Theorem W}

\numberwithin{equation}{section}

\def\N{{\mathbb N}}
\def\Z{{\mathbb Z}}

\def\K{{\mathbb K}}

\def\epsilon{\varepsilon}
\def\kappa{\varkappa}
\def\phi{\varphi}
\def\leq{\leqslant}
\def\geq{\geqslant}

\def\supp{\hbox{\tt supp}\,}
\def\dim{{\rm dim}\,}
\def\ssub#1#2{#1_{{}_{{\scriptstyle #2}}}}
\def\dimk{{\ssub{\dim}{\K}\,}}

\def\supp{\hbox{\tt supp}\,}

\title{Asymptotically optimal $k$-step nilpotency of quadratic algebras\\ and the Fibonacci
numbers}

\author{Natalia Iyudu and Stanislav Shkarin}

\date{}

\begin{document}

\maketitle

\begin{abstract}It follows from the Golod--Shafarevich theorem that
if $k\in\N$ and $R$ is an associative algebra given by $n$
generators and $d<\frac{n^2}{4}\cos^{-2}(\frac{\pi}{k+1})$ quadratic
relations, then $R$ is not $k$-step nilpotent. We show that the
above estimate is asymptotically optimal. Namely, for every
$k\in\N$, there is a sequence of algebras $R_n$ given by $n$
generators and $d_n$ quadratic relations such that $R_n$ is $k$-step
nilpotent and
$\lim\limits_{n\to\infty}\frac{d_n}{n^2}=\frac{1}{4}\cos^{-2}(\frac{\pi}{k+1})$.
\end{abstract}



\section{Introduction \label{s1}}\rm

Throughout this paper $\K$ is an arbitrary field, $\Z_+$ is the
set of non-negative integers and $\N$ is the set of positive
integers. For a set $X$, $\K\langle X\rangle$ stands for the free
associative algebra over $\K$ generated by $X$. We deal with
{\it quadratic algebras}, that is, algebras $R$ given as
$\K\langle X\rangle/I$, where $I$ is the ideal in $\K\langle X\rangle$
generated by a collection of homogeneous elements (called relations)
of degree $2$.

Algebras of this class, their growth, their Hilbert series and
nil/nilpotency properties have been extensively studied, see
\cite{popo,smo,ufnar} and references therein.  One of the most
challenging questions in the area (see \cite{smo,zelm}) is the Kurosh problem
of whether there is an infinite dimensional nil algebra in this class.
A version of this question dealing with algebras of finite Gelfand--Kirillov
dimension was solved in \cite{lesmo}. The Golod--Shafarevich type lower estimates
for the dimensions of the graded components of an algebra play a crucial role
in the study of quadratic algebras. These estimates have many other applications, for instance, to $p$-groups and class field theory \cite{gosh,zelm1}.

Recall that a $\K$-algebra $R$ defined by the set $X$ of generators and a set of
homogeneous relations inherits the degree grading from the free algebra $\K\langle X\rangle$.
If $X$ is finite, one can consider the Hilbert series of $R$:
\begin{equation*}%\label{HI}
H_R(t)=\sum_{q=0}^\infty (\dimk R_q)\,t^q,
\end{equation*}
where $R_q$ is the $q^{\rm th}$ homogeneous component of $R$.
The original Golod--Shafarevich theorem provides a lower estimate
for the coefficients of $H_R$. In the case of quadratic algebras the
theorem reads as follows \cite{gosh,popo}. For two power series $a(t)$ and
$b(t)$ with real coefficients we write $a(t)\geq b(t)$ if $a_j\geq b_j$
for any $j\in\Z_+$, while $|a(t)|$ stands for the power series obtained from $a(t)$ by
replacing by zeros all coefficients starting from the first
non-positive one.

\begin{thmgs}Let, $n\in\N$, $0\leq d\leq
n^2$ and $R$ be a quadratic $\K$-algebra with $n$ generators and $d$
relations. Then $H_R(t)\geq |(1-nt+dt^2)^{-1}|$.
\end{thmgs}

In particular, Theorem~GS provides a lower estimate on the order of
nilpotency of $R$.

\begin{definition}\label{kstep} A graded algebra $R$ is called
{\it $k$-step nilpotent} if $R_k=\{0\}$.
\end{definition}

Analysing the series $K(t)=|(1-nt+dt^2)^{-1}|$ in a standard way, one can
easily see that it is a polynomial of degree $<k$ if and only if
\begin{equation}\label{phik}
\frac d{n^2}\geq \phi_k,\ \ \text{where}\ \
\phi_k=\frac{1}{4}\cos^{-2}\Bigl(\frac{\pi}{k+1}\Bigr).
\end{equation}

For the sake of convenience, we outline the argument. If $(1-nt+dt^2)^{-1}=\sum\limits_{m=0}^\infty c_mt^m$
(the Taylor series expansion), then $K(t)$ is not a polynomial of degree $<k$ precisely when $c_m>0$ for $0\leq m\leq k$. Next, if $x^2-nx+d=(x-a)(x-b)$ ($a$ and $b$ are complex numbers in general), then an easy computation yields that $c_m=(m+1)(n/2)^m$ if $a=b$ and $c_m=\frac{a^{m+1}-b^{m+1}}{a-b}$ otherwise for $m\in\Z_+$. It follows that $c_m>0$ for all $m\in\Z_+$ if $a$ and $b$ are real, which happens precisely when $d\leq\frac{n^2}{4}$. If $n^2\geq d>\frac{n^2}{4}$, then $a,b=\sqrt{d}e^{\pm i\alpha}$, where $\alpha=\arccos\frac{n}{\sqrt d}$. Hence $c_m=\frac{a^{m+1}-b^{m+1}}{a-b}=d^{m/2}\frac{\sin(m+1)\alpha}{\sin\alpha}$ for $m\in\Z_+$. Clearly $c_m$ for $0\leq m\leq k$ are positive precisely when $(k+1)\alpha<\pi$. After plugging in $\alpha=\arccos\frac{n}{\sqrt d}$, (\ref{phik}) follows.

Formula (\ref{kstep}) together with Theorem~GS and the obvious fact that
the sequence $\{\phi_k\}$ decreases and converges to $\frac14$
implies the following corollary, which can be found in \cite{popo}.

\begin{corgs}If $R$ is a quadratic $\K$-algebra given by $n$
generators and $d<\phi_k n^2$ relations, then $\dim R_k>0$, where
$\phi_k$ is defined in $(\ref{phik})$. That is, $R$ is not $k$-step
nilpotent. In particular, if $d\leq\frac{n^2}{4}$, then $\dim R_k>0$
for every $k\in\N$ and therefore $R$ is infinite dimensional.
\end{corgs}

Asymptotic optimality of the last statement in Corollary~GS was
proved by Wisliceny \cite{wis}.

\begin{thmw}For every $n\in\N$, there exists a quadratic $\K$-algebra
$R$ given by $n$ generators and $d_n$ relations such that $R$ is
finite dimensional and
$\lim\limits_{n\to\infty}\frac{d_n}{n^2}=\frac14$.
\end{thmw}

More specifically, Wisliceny has constructed a quadratic algebra
given by $n$ generators and $\bigl\lceil
\frac{n^2+2n}{4}\bigr\rceil$ semigroup relations (that is, every
relation is either a degree 2 monomial or a difference of two degree
2 monomials), which is finite dimensional. Note that here and everywhere
below $\lfloor t\rfloor$ is the largest integer $\leq t$,
while $\lceil t\rceil$ is the smallest integer $\geq t$, where $t$
is a real number. The authors \cite{ns}
have improved the last result by showing that the minimal number of
semigroup quadratic relations needed for finite dimensionality of an
algebra with $n$ generators is exactly $\bigl\lceil
\frac{n^2+n}{4}\bigr\rceil$. The number $\bigl\lceil
\frac{n^2+1}{4}\bigr\rceil$ remains a conjectural answer to the same
question in the class of general quadratic (not necessarily
semigroup) algebras.

\subsection{Results}

Note that if $R$ is $k$-step nilpotent, then $R_m=\{0\}$ for $m\geq
k$ and therefore $R$ is finite dimensional provided $|X|<\infty$,
where $X$ is the set of generators of $R$.
Thus $R$ is $k$-step nilpotent if and only if $H_R$ is a polynomial
of degree $<k$.

In this article we show that the first statement in Corollary~GS is
asymptotically optimal for every $k\geq 2$. In order to formulate
the exact statement, we shall introduce the following numbers. For
$n\in\N$ and $k\geq 2$ let
\begin{equation}\label{dnk}
d_{n,k}=\min_{n=a_1+{\dots}+a_{k-1}}\ \max_{1\leq j\leq
k-1}(a_1+{\dots}+a_j)(a_j+{\dots}+a_{k-1}),
\end{equation}
where $a_j$ are assumed to be non-negative integers. It turns out
that the integers $d_{n,k}$ are not too far from $\phi_k n^2$.

\begin{lemma}\label{dnph} For each $n,k\in\N$ with $k\geq 2$,
\begin{equation}\label{estim}
\textstyle
\phi_k n^2\leq d_{n,k}\leq \phi_k n^2+\frac{(1+\phi_k)n}{2}+\frac14.
\end{equation}
In particular, $\lim\limits_{n\to\infty}\frac{d_{n,k}}{\phi_k n^2}=1$ for each
$k\geq 2$.
\end{lemma}

We have defined the numbers $d_{n,k}$ since they feature in the
following theorem.

\begin{theorem}\label{asop}Let $k\geq2$.
Then for every $n\in\N$, there exists a quadratic $\K$-algebra $R$
given by $n$ generators and $d_{n,k}$ relations such that $R$ is
$k$-step nilpotent.
\end{theorem}

Corollary~GS, Theorem~\ref{asop} and Lemma~\ref{dnph} imply that the
first statement in Corollary~GS is asymptotically optimal. Note that
Anick \cite{ani1,ani2} conjectured that for any $n\in\N$ and $0\leq
d\leq n^2$, there is a quadratic $\K$-algebra $R$ with $n$
generators and $d$ relations such that $H_R(t)=|(1-nt+dt^2)^{-1}|$.
The problem whether this conjecture is true remains open.
Theorem~\ref{asop} can be considered as an affirmative solution of
its natural asymptotic version. It is also worth noting that for
$k=2$, the statement of Theorem~\ref{asop} is trivial, while the
case $k=3$ was done by Anick \cite{ani1}. It is also worth
mentioning that the asymptotic optimality of the first statement in
Corollary~GS for $k=4$ and for $k=5$ in the case $|\K|=\infty$ was
earlier obtained by the authors \cite{na}
building upon the ideas set in \cite{cana} and using a completely
different approach. We refer to \cite{asyy} for a result on asymptotic
optimality of Theorem~GS in a completely different sense.

Curiously enough, for some pairs $(n,k)$ the estimate provided by
Theorem~\ref{asop} hits the mark. We illustrate this observation by
the following result dealing with the cases $k=4$ and $k=5$. Note
that $\phi_4=\frac{3-\sqrt 5}{2}$ and $\phi_5=\frac13$. Recall that
Fibonacci numbers are the members of the recurrent sequence defined
by $F_0=F_1=1$ and $F_n=F_{n-1}+F_{n-2}$ for $n\geq2$.

\begin{theorem}\label{k45} The equality
$d_{n,4}=\bigl\lceil\frac{3-\sqrt 5}{2}n^2\bigr\rceil$ holds if and
only if $n$ is a Fibonacci number. The equality
$d_{n,5}=\bigl\lceil\frac{n^2}{3}\bigr\rceil$ holds if and only if
$n\in\{1,2\}$ or $n$ is divisible by $6$.
\end{theorem}

Note that Theorem~\ref{k45}, Theorem~\ref{asop} and Corollary~GS
imply that if $k=4$ and $n$ is a Fibonacci number or if $k=5$ and
$6$ divides $n$, then the minimal number of quadratic relations
needed for the finite dimensionality of an algebra with $n$
generators is exactly $\lceil \phi_k n^2\rceil$. The proof of
Theorem~\ref{asop} is based upon the following general result. We
start by introducing some notation.

\begin{definition}\label{paror}Let $X$ be the union of pairwise
disjoint sets $A_1,\dots,A_k$ and
\begin{equation}\label{M}
M=M(A_1,\dots,A_k)=\bigcup_{1\leq j\leq q\leq n}\!\!A_q\times
A_j\subseteq X\times X.
\end{equation}
We introduce the following {\it partial ordering on $M$, generated
by the partition} $\{A_1,\dots,A_k\}$. Namely, for distinct elements
$(a,b)$ and $(c,d)$ of $M$, we write $(a,b)\prec(c,d)$ if $(a,b)\in
A_l\times A_j$ and $(c,d)\in A_m\times A_r$ with $m\geq r>l\geq j$.
\end{definition}

\begin{definition}\label{supp} For a homogeneous degree $2$
polynomial $g$ in the free algebra $\K\langle X\rangle$, the
$($uniquely determined$)$ finite subset $S$ of $X\times X$ such that
$g=\sum\limits_{(x,y)\in S} c_{x,y}xy$ with
$c_{x,y}\in\K\setminus\{0\}$ is called the {\it support} of $g$ and
is denoted $S=\supp(g)$.
\end{definition}

The next result is one of the main tools in the proof of
Theorem~\ref{asop}.

\begin{theorem}\label{chin} Let $k\in\N$, $\{A_1,\dots,A_k\}$ be
a partition of a set $X$ and $M$ be the set defined in $(\ref{M})$.
Assume also that $\{f_\alpha\}_{\alpha\in \Lambda}$ is a family of
homogeneous degree $2$ elements of the free algebra $\K\langle
X\rangle$ such that
$\bigcup\limits_{\alpha\in\Lambda}\supp(f_\alpha)=M$ and each
$\supp(f_\alpha)$ is a chain in $M$ with respect to the partial
ordering $\prec$ on $M$, generated by the partition
$\{A_1,\dots,A_k\}$ as in Definition~$\ref{paror}$. Then the algebra $R=\K\langle X\rangle/I$ with
$I={\tt Id}\{f_\alpha:\alpha\in\Lambda\}$ is $(k+1)$-step nilpotent.
\end{theorem}

We conclude the introduction by providing a specific example of an
application of Theorem~\ref{chin}.

\begin{example} \label{ex8} Let $X=\{a,b,c,p,q,x,y,z\}$ be an
$8$-element set partitioned into $3$ subsets $A_1=\{a,b,c\}$,
$A_2=\{p,q\}$ and $A_3=\{x,y,z\}$. Let $M$ and the partial
ordering $\prec$ on $M$ be as in Definition~{\rm\ref{paror}}. Consider
the following $25$ quadratic relations:
\begin{align*}
f_1&=xc,&f_2&=xa,&f_3&=xp+ab,&f_4&=yz+qc,\qquad
f_5=pq,
\\
f_6&=yc,&f_7&=ya,&f_8&=yp+bb,&f_9&=yy+qb,
\\
f_{10}&=zc,&f_{11}&=za,&f_{12}&=zp+cb,&f_{13}&=yx+qa,
\\
f_{14}&=xb,&f_{15}&=xq+ac,&f_{16}&=xz+pc,&f_{17}&=zz+qq+ca,
\\
f_{18}&=yb,&f_{19}&=yq+bc,&f_{20}&=xy+pb,&f_{21}&=zy+qp+ba,
\\
f_{22}&=zb,&f_{23}&=zq+cc,&f_{24}&=xx+pa,&f_{25}&=zx+pp+aa.
\end{align*}
It is straightforward to verify that the support of each $f_j$ is a
chain in $(M,\prec)$ and that the union of $\supp(f_j)$ for $1\leq
j\leq 25$ is $M$. Theorem~{\rm\ref{chin}} ensures that the algebra given
by the $8$-element generator set $X$ and the relations $f_j$
with $1\leq j\leq 25$ is $4$-step nilpotent. Incidentally,
$25=\bigl\lceil\phi_4\cdot8^2\bigr\rceil$, which means $($see
Corollary~{\rm GS)} that a quadratic algebra given by $8$ generators and
$\leq24$ relations is never $4$-step nilpotent.
\end{example}

\section{Combinatorial lemmas}

Theorem~\ref{chin} allows us to construct $k$-step nilpotent
quadratic algebras with few relations. In order to do this, we need
an estimate on the number of relations in an algebra featuring in
Theorem~\ref{chin}. Recall that the {\it width} $w(X,<)$ of a
partially ordered set $(X,<)$ is the supremum of the cardinalities
of antichains in $X$.

\begin{lemma}\label{esti} Let $k\in\N$, $\{A_1,\dots,A_k\}$ be
a partition of a finite set $X$ and $M\subseteq X^2$ be the set
defined in $(\ref{M})$ with the partial ordering $\prec$ introduced
in Definition~$\ref{paror}$. For $1\leq q\leq k$, let
$B_q=\bigcup\limits_{j\geq q\geq m}A_j\times A_m$. Then
$w(M,\prec)=\max\{|B_1|,\dots,|B_k|\}$.
\end{lemma}

\begin{proof} It is a straightforward exercise to verify
that each $B_q$ is an antichain in $(M,\prec)$ and that every
antichain is contained in at least one of the sets $B_q$.
\end{proof}

We also need the following observation.

\begin{lemma}\label{alpk} Let $k\geq 2$ and
$\alpha_0,\alpha_1,\dots,\alpha_{k-1}\geq 0$ be defined by the
fromulae $\alpha_0=0$, $\alpha_1=\phi_k$ and
$\alpha_j=\frac{\phi_k}{1-\alpha_{j-1}}$ for $2\leq j\leq k-1$. Then
\begin{align}
&\text{$0=\alpha_0<\alpha_1<{\dots}<\alpha_{k-1}=1$}, \label{alpk00}
\\
&\text{$\alpha_j(1-\alpha_{j-1})=\phi_k$ for $1\leq j\leq k-1$} \label{alpk01}
\\
&\text{and $\max\limits_{1\leq j\leq k-1}(\alpha_j-\alpha_{j-1})=\phi_k$
$($attained for $j=1$ and for $j=k-1)$.}\label{alpk02}
\end{align}
\end{lemma}

\begin{proof} Obviously, (\ref{alpk01}) is a direct consequence  of the definition of $\alpha_j$. Next, (\ref{alpk02}) follows easily from (\ref{alpk00}). Indeed, assuming that (\ref{alpk00}) holds, we have $\alpha_{k-1}=1$, which implies $\alpha_{k-2}=1-\phi_k$. Since $\alpha_j-\alpha_{j-1}=\frac{\phi_k}{1-\alpha_{j-1}}-\alpha_{j-1}$ and $0\leq \alpha_{j-1}\leq 1-\phi_k$ for $1\leq j\leq k-1$, (\ref{alpk02}) follows from the elementary fact that the function $\frac{\phi_k}{1-x}-x$ on the interval $[0,1-\phi_k]$ attains its maximal value at the end-points.

Thus it remains to verify (\ref{alpk00}). For $0<t\leq 1$ consider the rational function $f_t(x)=\frac{t}{1-x}$ and for $m\in\Z_+$ let $f_t^{[m]}$ be the $m^{\rm th}$ iterate of $f_t$: $f_t^{[0]}(x)=x$ and $f_t^{[m]}=f_t\circ{\dots}\circ f_t$ $m$ times for $m\in\N$. We start with an elementary observation

\begin{equation}\label{tlqu1}
\begin{array}{l}
\text{if $0\leq t\leq\frac14$, then the sequence $\{f_t^{[m]}(0)\}_{m\in\Z_+}$ is strictly increasing}
\\
\text{and converges to the fixed point $w_t=\frac{1-\sqrt{1-4t}}{2}\in\bigl[0,\frac12\bigr]$ of $f_t$.}
\end{array}
\end{equation}
For instance, to justify (\ref{tlqu1}), one can use induction with respect to $m$ to prove the chain of inequalities $0\leq f_t^{[m]}(0)<f_t^{[m+1]}(0)<w_t$.

Next, it is easy to verify that if $\frac14<t\leq 1$, then $f_t(x)>x$ for $x\in[0,1)$. Hence,
\begin{equation}\label{tlqu2}
\text{$f_t^{[m+1]}(0)>f_t^{[m]}(0)$ provided $0\leq f_t^{[m]}(0)<1$}.
\end{equation}
For each $m\in\Z_+$, we consider the rational function $h_m(t)=f_t^{[m]}(0)$ of the variable $t$. Now we observe that (\ref{alpk02}) follows from the claim
\begin{equation}\label{tlqu3}
\text{for every $m\in\N$, $\phi_{m+1}$ is the smallest solution of the equation $h_m(t)=1$ on
$\textstyle\bigl(\frac14,1\bigr]$.}
\end{equation}
Indeed, assume that (\ref{tlqu3}) holds. By (\ref{tlqu1}), $0<h_m(t)<\frac12$ for every $m\in\N$ and $t\in\bigl(0,\frac14\bigr]$. Since the sequence $\{\phi_m\}$ is decreasing, $h_j(t)<1$ whenever $j\leq m$ and $0\leq t< \phi_{m+1}$. Using (\ref{tlqu3}) with $m=k-1$ and (\ref{tlqu2}), we now have
$$
0=f^{[0]}_{\phi_k}(0)<f^{[1]}_{\phi_k}(0)<{\dots}<f^{[k-1]}_{\phi_k}(0)=h_{k-1}(\phi_k)=1.
$$
On the other hand, by definition of $\alpha_j$, $\alpha_j=f_{\phi_k}^{[j]}(0)$ for $0\leq j\leq k-1$ and (\ref{alpk02}) follows.

Thus it remains to prove (\ref{tlqu3}). Using the obvious recurrent relation $h_{j+1}(t)=\frac{t}{1-h_j(t)}$ together with the initial data $h_0=0$, one can use the induction with respect to $m$ to verify that
$$
\textstyle h_m(t)=t\frac{a^m-\overline{a}^m}{a^{m+1}-\overline{a}^{m+1}}\ \ \text{for $m\in\Z_+$ and $t\in \bigl[\frac14,1\bigr]$, where $a=a(t)=\frac{1+i\sqrt{4t-1}}{2}$}.
$$
Hence for $t\in\bigl[\frac14,1\bigr]$,
\begin{equation}\label{tlqu4}
\textstyle h_m(t)=1\iff (a/\overline{a})^m=(\overline{a}-t)/(a-t)\iff
e^{im\alpha(t)}=e^{i\beta(t)},
\end{equation}
where
$$
\textstyle\alpha(t)=2\arccos\frac1{2\sqrt t} \ \ \text{and}\ \ \beta(t)=2\pi-2\arccos\bigl(\frac1{2t}-1\bigr)
$$
are the arguments of the unimodular complex numbers $a/\overline{a}$ and $(\overline{a}-t)/(a-t)$. The case $m=1$ is trivial. Assuming that $m\geq 2$ and using (\ref{tlqu4}), we see that the smallest $t\in \bigl[\frac14,\frac12\bigr]$ satisfying $h_m(t)=1$ must satisfy $m\alpha(t)=\beta(t)$. Since the function $m\alpha(t)-\beta(t)$ on the interval $\bigl[\frac14,\frac12\bigr]$ is strictly increasing (look at the derivative) and has values of opposite signs at the ends, there is exactly one $t_m\in \bigl[\frac14,\frac12\bigr]$ satisfying $m\alpha(t_m)=\beta(t_m)$. Then $t_m$ is the smallest solution of the equation $h_m(t)=t$ on the interval $\bigl[\frac14,1\bigr]$. Since $\phi_{m+1}\in \bigl[\frac14,\frac12\bigr]$, (\ref{tlqu3}) will follow if we show that $m\alpha(\phi_{m+1})=\beta(\phi_{m+1})$. This is indeed true: plugging in $\phi_{m+1}=\frac1{4\cos^2(\pi/(m+2))}$, we have
\begin{align*}
&\textstyle m\alpha(\phi_{m+1})=2m\arccos\bigl(\cos\bigl(\frac{\pi}{m+2}\bigr)\bigr)=\frac{2\pi m}{m+2};
\\
&\textstyle\beta(\phi_{m+1})=2\pi-2\arccos\bigl(2\cos^2\bigl(\frac{\pi}{m+2}\bigr)-1\bigr)=
2\pi-2\arccos\bigl(\cos\bigl(\frac{2\pi}{m+2}\bigr)\bigr)=2\pi-\frac{4\pi}{m+2}=\frac{2\pi m}{m+2}.
\end{align*}
Hence $m\alpha(\phi_{m+1})=\beta(\phi_{m+1})$, which completes the proof.
\end{proof}

\section{Proof of Theorem~\ref{chin}}

For $k\in\N$, we denote $\N_k=\{1,2,\dots,k\}$.
Assume the contrary. Then the set $\Omega$ of
$j=(j_1,\dots,j_{k+1})\in\N_k^{k+1}$ such that there are $x_1\in
A_{j_1}$, $\dots$, $x_{k+1}\in A_{j_{k+1}}$ for which $x_1\dots
x_{k+1}\notin I$ is non-empty. We endow $\N_k^{k+1}$ with the
lexicographical ordering $<$ counting from the right-hand side. That
is, $j<m$ if and only if there is $l\in\N_{k+1}$ such that $j_l<m_l$
and $j_r=m_r$ for $r>l$. Since $<$ is a total ordering on the finite
set $\N_k^{k+1}$ and $\Omega\subseteq \N_k^{k+1}$ is non-empty,
$\Omega$ has a unique element $j$ minimal with respect to $<$. Since
$j\in\Omega$, there are $x_1\in A_{j_1}$, $\dots$, $x_{k+1}\in
A_{j_{k+1}}$ for which $x_1\dots x_{k+1}\notin I$.

Now we shall construct inductively $m_1,\dots,m_{k+1}\in\N_{k}$ and
monomials $u_1,\dots,u_{k+1}$ in $\K\langle X\rangle$ of degree
$k+1$ such that
\begin{align}
&\text{$m_l>m_{l-1}$ if $l\geq 2$}; \label{ind1}
\\
&u_l\notin I; \label{ind2}
\\
&\text{$u_l=v_lw_lx_{l+1}x_{l+2}\dots x_{k+1}$, where $w_l\in
A_{m_l}$ and $v_l$ is a monomial of degree $l-1$}. \label{ind3}
\end{align}

We start by setting $u_1=x_1\dots x_{k+1}$ and $m_1=j_1$ and
observing that (\ref{ind1}--\ref{ind3}) with $l=1$ are satisfied.
Assume now that $2\leq l\leq k+1$ and that $m_1,\dots,m_{l-1}$ and
$u_1,\dots,u_{l-1}$ satisfying the desired conditions are already
constructed.

If $m_{l-1}<j_l$, then we set $m_l=j_l$, $w_l=x_l$, $u_l=u_{l-1}$
and $v_l=v_{l-1}w_{l-1}$. Using the induction hypothesis, we see
that (\ref{ind1}--\ref{ind3}) are satisfied. It remains to consider
the case $m_{l-1}\geq j_l$. In this case $w_{l-1}x_l\in M$ and
therefore there is $\alpha\in\Lambda$ such that
$(w_{l-1},x_l)\in\supp(f_\alpha)$. Let
$S=\supp(f_\alpha)\setminus\{(w_{l-1},x_l)\}$. Since $f_\alpha\in
I$,
$$
w_{l-1}x_l=\sum_{(a,b)\in S} c_{a,b}ab\ \ (\bmod I)\ \ \ \text{with
$c_{a,b}\in\K$.}
$$
Using (\ref{ind3}) for $l-1$ and the above display, we get
$$
u_{l-1}=\sum_{(a,b)\in S}c_{a,b}v_{l-1}abx_{l+1}\dots x_{k+1}\ \
(\bmod I).
$$
Since  $\supp(f_\alpha)$ is a chain in $M$ with respect to $\prec$,
for every $(a,b)\in S$, either $(a,b)\prec (w_{l-1},x_l)$ or
$(w_{l-1},x_l)\prec (a,b)$. If $(a,b)\prec (w_{l-1},x_l)$, $b$ is
contained in $A_q$ with $q<j_l$. Using the definition of $\Omega$
and the minimality of $j$ in $\Omega$, we obtain
$$
v_{l-1}abx_{l+1}\dots x_{k+1}\in I\ \ \text{if $(a,b)\in S$,
$(a,b)\prec (w_{l-1},x_l)$.}
$$
According to the last two displays
$$
u_{l-1}=\sum_{(a,b)\in S\atop (w_{l-1},x_l)\prec
(a,b)}c_{a,b}v_{l-1}abx_{l+1}\dots x_{k+1}\ \ (\bmod I).
$$
By (\ref{ind2}) for $l-1$, $u_{l-1}\notin I$. Thus, using the above
display, we can pick $(a,b)\in S$ such that $(w_{l-1},x_l)\prec
(a,b)$ and $v_{l-1}abx_{l+1}\dots x_{k+1}\notin I$. Now we set
$u_l=v_{l-1}abx_{l+1}\dots x_{k+1}$, $w_l=b$, $v_l=v_{l-1}a$ and
take $m_l$ such that $w_l=b\in A_{m_l}$.

Since $w_{l-1}\in A_{m_{l-1}}$ and $(w_{l-1},x_l)\prec
(a,b)=(a,w_l)$, we have $m_l>m_{l-1}$. Thus (\ref{ind1}--\ref{ind3})
are satisfied. This completes the inductive procedure of
constructing $m_1,\dots,m_{k+1}$ and $u_1,\dots,u_{k+1}$. By
(\ref{ind1}), $m_j$ for $1\leq j\leq k+1$ are $k+1$ pairwise
distinct elements of the $k$-element set $\N_k$. We have arrived to
a contradiction, which proves that $R$ is $(k+1)$-step nilpotent.

\section{Proofs of Theorem~\ref{asop} and Lemma~\ref{dnph}}

Let $k\geq2$, $n\in\N$ and $a_1,\dots,a_{k-1}\in\Z_+$ be such that
$a_1+{\dots}+a_{k-1}=n$. In order to prove Theorem~\ref{asop}, it
suffices to prove that there is a quadratic $\K$-algebra $R$ given
by $n$ generators and
$$
d=\max_{1\leq j\leq k-1}(a_1+{\dots}+a_j)(a_j+{\dots}+a_{k-1})
$$
relations such that $R$ is $k$-step nilpotent.

Let $X$ be an $n$-element set of generators. Since
$a_1+{\dots}+a_{k-1}=n$, we can present $X$ as the union of the
pairwise disjoint sets $A_1,\dots,A_{k-1}$ with $|A_j|=a_j$ for
$1\leq j\leq k-1$. Consider the set $M\subset X^2$ defined in
(\ref{M}) and the partial ordering $\prec$ on $M$ generated by the
partition $\{A_1,\dots,A_{k-1}\}$. For $1\leq j\leq k-1$, let
$B_j=\bigcup\limits_{q\geq j\geq m}A_q\times A_m$. Clearly,
$|B_j|=(a_1+{\dots}+a_j)(a_j+{\dots}+a_{k-1})$. Hence
$d=\max\{|B_1|,\dots,|B_{k-1}|\}$. By Lemma~\ref{esti},
$w(M,\prec)=d$. According to the Dilworth theorem (see \cite{gal}
for a short inductive proof) the width of a finite partially ordered
set $P$ is precisely the minimal number of chains needed to cover
$P$. Hence, we can write $M=\bigcup\limits_{q=1}^d C_q$, where each
$C_q$ is a chain in $M$. Now we consider the homogeneous degree 2
elements of $\K\langle X\rangle$ given by
$$
f_q=\sum_{(a,b)\in C_q}ab\ \ \text{for $1\leq q\leq d$}.
$$
Clearly $\supp(f_q)=C_q$. Thus the union of the supports of $f_q$ is
$M$ and each $\supp(f_q)$ is a chain in $M$. By Theorem~\ref{chin},
the algebra $R$ given by the relations $f_q$ for $1\leq q\leq d$ is
$k$-step nilpotent. This completes the proof of Theorem~\ref{asop}.

Now we shall prove Lemma~\ref{dnph}. By Theorems~GS and~\ref{asop},
$d_{n,k}\geq \phi_k n^2$ for every $k\geq 2$ and $n\in\N$. This
proves the first inequality in (\ref{estim}). It remains to prove
the second one. By Lemma~\ref{alpk}, there are
$\alpha_0,\dots,\alpha_{k-1}\in[0,1]$ such that
$0=\alpha_0<\alpha_1<{\dots}<\alpha_{k-1}=1$ and
$\alpha_j(1-\alpha_{j-1})=\phi_k$ for $1\leq j\leq k-1$. Now for
$0\leq j\leq k-1$ let $b_j=\lceil n\alpha_j-\frac12\rceil$. Clearly
$0=b_0\leq b_1\leq{\dots}\leq b_{k-1}=n$. Now we set
$a_j=b_j-b_{j-1}$ for $1\leq j\leq k-1$. Then $a_j\in\Z_+$ and
$a_1+{\dots}+a_{k-1}=n$. Hence
$$
d_{n,k}\leq\max_{1\leq j\leq
k-1}(a_1+{\dots}+a_j)(a_j+{\dots}+a_{k-1})=\!\!\max_{1\leq j\leq
k-1} b_j(n-b_{j-1})=\!\!\max_{1\leq j\leq k-1}\bigl\lceil
n\alpha_j-{\textstyle \frac12}\bigr\rceil \cdot \bigl\lfloor
n(1-\alpha_{j-1})+{\textstyle \frac12}\bigr\rfloor.
$$
It is easy to see that for every $\alpha,\beta\in[0,1]$,
$$
\textstyle\bigl\lceil n\alpha-\frac12\bigr\rceil \cdot \bigl\lfloor
n\beta+\frac12\bigr\rfloor -\alpha\beta n^2\leq
\frac{\alpha+\beta}{2}n+\frac14.
$$
From the last two displays and the equalities
$\alpha_j(1-\alpha_{j-1})=\phi_k$ it follows that
$$
d_{n,k}\leq \phi_k n^2+\frac n2\max_{1\leq j\leq
k-1}(1+\alpha_j-\alpha_{j-1})+\frac14.
$$
By Lemma~\ref{alpk}, the maximum in the above display equals
$\phi_k$. Thus $d_{n,k}\leq \phi_k n^2+\frac{1+\phi_k}{2}n+\frac14$,
which completes the proof of Lemma~\ref{dnph}.

\section{4-Step nilpotency and the Fibonacci numbers}

First, we derive an explicit formula for $d_{n,4}$.

\begin{lemma}\label{dn4} For every $n\in\N$,
\begin{equation}\label{dn4f}
d_{n,4}=\min\bigl\{\bigl\lceil {\textstyle
\frac{\sqrt5-1}{2}}n\bigr\rceil^2,n\bigl\lceil {\textstyle
\frac{3-\sqrt5}{2}}n\bigr\rceil\bigr\}.
\end{equation}
\end{lemma}

\begin{proof} Using (\ref{dnk}) with $k=4$ and denoting $a=a_1$ and
$b=a_3$, we obtain
$$
d_{n,4}=\min\{\max\{na,nb,(n-a)(n-b)\}:a,b\in\Z_+,\ a+b\leq n\}.
$$
An obvious symmetry consideration yields
$$
d_{n,4}=\min\{\max\{na,nb,(n-a)(n-b)\}:a,b\in\Z_+,\ b\leq a,\
a+b\leq n\}.
$$
Since $nb\leq na$ and $(n-a)(n-b)\geq (n-a)^2$ when $a,b\in\Z_+$
satisfy $b\leq a\leq n$, we have
\begin{equation}\label{dn4eq}
d_{n,4}=\min\{\max\{na,(n-a)^2\}:a\in\Z_+,\ 2a\leq n\}.
\end{equation}

Now, assume that $a\in\Z_+$ satisfies $2a\leq n$. Solving a
quadratic inequality we see that $na\geq (n-a)^2$ holds precisely
when $a\geq \phi_4n$. Hence (\ref{dn4eq}) can be rewritten as
$$
\begin{array}{l}d_{n,4}=\min\{a_n,b_n\},\ \ \text{where}\\
a_n=\min\{na:a\in\Z_+,\ \phi_4n\leq a\leq n/2\}\ \ \text{and}\ \
b_n=\min\{(n-a)^2:a\in\Z_+,\ a\leq \phi_4n\}.
\end{array}
$$
Clearly, the minimum in the definition of $a_n$ is attained for
$a=\lceil\phi_4 n\rceil$ and the minimum in the definition of $b_n$
is attained for $a=\lfloor \phi_4n\rfloor$.  Hence
$a_n=n\lceil\phi_4 n\rceil$ and $b_n=\lceil(1-\phi_4)n\rceil^2$.
Using the equalities $\phi_4=\frac{3-\sqrt 5}{2}$ and
$1-\phi_4=\frac{\sqrt 5-1}{2}$, we see that (\ref{dn4f}) follows
from the above display.
\end{proof}

\begin{corollary}\label{ttt} The equality $d_{n,4}=\bigl\lceil\phi_4
n^2\bigr\rceil$ holds if and only if either $\bigl\lceil\phi_4
n^2\bigr\rceil$ is divisible by $n$ or $\bigl\lceil\phi_4
n^2\bigr\rceil$ is a square of a positive integer.
\end{corollary}

\begin{proof} Let $m=\bigl\lceil\phi_4 n^2\bigr\rceil$. From Lemma~\ref{dn4}
it follows that $d_{n,4}$ is always either divisible by $n$ or is a
square. Thus the equality $m=d_{n,4}$ can only hold if either $m$ is
divisible by $n$ or $m$ is a square.

If $m$ is divisible by $n$, we can write $m=nj$ for some $j\in\N$.
Now it is easy to see that $j=\bigl\lceil \frac{3-\sqrt
5}{2}n\bigr\rceil$ and therefore, by Lemma~\ref{dn4}, $d_{n,4}\geq
jn=m$. On the other hand, choosing $a=j$ and using (\ref{dn4eq}), we
get $d_{n,4}\leq \max\{nj,(n-j)^2\}=nj$. Thus $d_{n,4}=nj=m$.

If $m$ is a square, we can write $m=j^2$ for some $j\in\N$. Now it
is easy to see that $j=\bigl\lceil \frac{\sqrt 5-1}{2}n\bigr\rceil$
and therefore, by Lemma~\ref{dn4}, $d_{n,4}\geq j^2=m$. On the other
hand, choosing $a=n-j$ and using (\ref{dn4eq}), we get $d_{n,4}\leq
\max\{n(n-j),j^2\}=j^2$. Thus $d_{n,4}=j^2=m$.
\end{proof}

\begin{proof}[Proof of the first part of Theorem~{\rm\ref{k45}}]
Let $F_0,F_1,\dots$ be the Fibonacci sequence and
$\phi=\frac{\sqrt5+1}2$ be the golden ratio number. Using the
formula $F_n=\frac{\phi^n-(-\phi)^{-n}}{\sqrt 5}$ together with the
equality $\phi_4=\phi^{-2}$, one can easily verify that $\bigl\lceil
\phi_4 F_k^2\bigr\rceil=F_{k-1}^2$ if $k$ is odd and $\bigl\lceil
\phi_4 F_k^2\bigr\rceil=F_kF_{k-2}$ if $k$ is even. Thus if $n$ is a
Fibonacci number, then $\bigl\lceil \phi_4 n^2\bigr\rceil$ is either
divisible by $n$ or is a square.

To show the converse, we use the following criterion of recognizing
the Fibonacci numbers due to M\"obius \cite{mob}. It says that a
positive integer $n$ is a Fibonacci number if and only if the
interval $(\phi n-n^{-1},\phi n+n^{-1})$ contains an integer.
Furthermore, if $m$ is an integer belonging to $(\phi n-n^{-1},\phi
n+n^{-1})$, then $m$ is the next Fibonacci number after $n$.

First, assume that $n\in\N$ and $\bigl\lceil \phi_4 n^2\bigr\rceil$
is divisible by $n$. Then $\phi_4 n^2+\theta=nk$, where $k\in\N$ and
$0<\theta<1$. Since $\phi_4=2-\phi$, it follows that $\phi
n-(2n-k)=\frac{\theta}n$ and therefore $2n-k\in (\phi n-n^{-1},\phi
n+n^{-1})$. By the criterion of M\"obius, $n$ is a Fibonacci number.
Finally, assume that $\bigl\lceil \phi_4 n^2\bigr\rceil$ is a square
number. Since $\phi_4=\phi^{-2}$, this means that
$\frac{n^2}{\phi^2}+\theta=k^2$, where $k\in\N$ and $0<\theta<1$. It
immediately follows that $k=\bigl\lceil \frac{n}{\phi}\bigr\rceil$.
In other words $k=\frac{n}{\phi}+\alpha$ with $0<\alpha<1$. Squaring
the last equality, we get
$k^2=\frac{n^2}{\phi^2}+\theta=\frac{n^2}{\phi^2}+\frac{2n\alpha}{\phi}+\alpha^2$.
In particular, $\frac{2n\alpha}{\phi}<\theta<1$. Hence
$\phi\alpha<\frac{\phi^2}{2n}$. Thus the equality
$k=\frac{n}{\phi}+\alpha$ implies $n=\phi k-\phi\alpha$ and
$$
\phi\alpha<\frac{\phi^2}{2n}=\frac{\phi^2}{2(\phi k-\phi\alpha)}<
\frac{\phi^2}{2(\phi k-\phi^2/2n)}.
$$
Since $n\geq k$, we have
$$
\phi\alpha<\frac{\phi^2}{2(\phi k-\phi^2/2k)}<\frac1k,
$$
where the last inequality is satisfied for $k>2$. Now the above
display and the equality $n=\phi k-\phi\alpha$ imply that $n$
belongs to the interval $(\phi k-k^{-1},\phi k+k^{-1})$. By the
criterion of M\"obius, both $k$ and $n$ are Fibonacci numbers provided $k>2$.
If $k=1$ or $k=2$, a direct computation yields $n=2$ or $n=3$ respectively, which
are Fibonacci numbers as well.

Thus we have proven that $\bigl\lceil \phi_4 n^2\bigr\rceil$ is
either divisible by $n$ or is a square number precisely when $n$ is
a Fibonacci number. By Lemma~\ref{ttt}, $d_{n,4}=\bigl\lceil \phi_4
n^2\bigr\rceil$ if and only if $n$ is a Fibonacci number.
\end{proof}

 \section{5-Step nilpotency}

In this section we prove the second part of Theorem~\ref{k45}. As in
the previous section we start by simplification the formula defining
$d_{n,5}$.

\begin{lemma}\label{D5} If $n\in\N$ is even, then $d_{n,5}=\frac n2\bigl\lceil
\frac{2n}{3}\bigr\rceil$. If $n\in\N$ is congruent to $-1$ modulo
$6$, then $d_{n,5}=n\bigl\lceil\frac{n(n+1)}{3n+1}\bigr\rceil$. If
$n\in\N$ is congruent to $1$ or to $3$ modulo $6$, then
$d_{n,5}=\frac{n+1}{2} \bigl\lceil\frac{2n^2}{3n+1}\bigr\rceil$.
\end{lemma}

\begin{proof} Using the symmetry in (\ref{dnk}) with respect to
reversing the order of $a_j$, we have
\begin{equation}\label{dn5}
\begin{array}{l}d_{n,5}=\min\{S(a):a\in\Z_+^4,\ a_1+a_2+a_3+a_4=n,\ a_1\leq
a_4\},\ \ \text{where}\\
S(a)=\max\{na_1,na_4,(a_1+a_2)(a_2+a_3+a_4),(a_1+a_2+a_3)(a_3+a_4)\}.
\end{array}
\end{equation}
It is easy to see that the minimum in (\ref{dn5}) can not be
attained when $a_2=0$ if $n>1$ (the case $n=1$ is trivial anyway).
If $a_1<a_4$ and $a_2>0$, one can easily check that $S(a')\leq
S(a)$, where $a'$ is obtained from $a$ by increasing $a_1$ by $1$
with simultaneous decreasing of $a_2$ by $1$. Similarly, if
$a_1=a_4$ and $|a_2-a_3|>1$, $S(a')\leq S(a)$, where $a'$ is
obtained from $a$ by increasing the smaller of $a_2$ and $a_3$ by
$1$ with simultaneous decreasing of the bigger one by $1$. It
follows that among $a\in\Z_+^4$ for which the minimum in (\ref{dn5})
is attained there must be at least one point satisfying $a_1=a_4$
and $|a_2-a_3|\leq 1$. Thus the minimum in (\ref{dn5}) is attained
at a point $a$ of the shape $a=(\alpha,\beta,\beta,\alpha)$ if $n$
is even and it is attained at a point $a$ of the shape
$a=(\alpha,\beta+1,\beta,\alpha)$ if $n$ is odd. Substituting this
data into (\ref{dn5}), we get
\begin{equation}\label{d5even}
d_{n,5}=\frac n2\min\{\max\{2a,n-a\}:a\in\Z_+,\ a\leq n/2\}\ \
\text{if $n$ is even}
\end{equation}
and
\begin{equation}\label{d5odd}
d_{n,5}=\min\{\max\{na,(n+1)(n-a)/2\}:a\in\Z_+,\ a\leq n/2\}\ \
\text{if $n$ is odd}.
\end{equation}
Since $\max\{2a,n-a\}=n-a$  if $3a\leq n$ and $\max\{2a,n-a\}=2a$ if
$3a\geq n$, (\ref{d5even}) implies that
$d_{n,5}=\min\bigl\{n\bigl\lceil\frac{n}3\bigr\rceil,\frac
n2\bigl\lceil\frac{2n}{3}\bigr\rceil\bigr\}=\frac
n2\bigl\lceil\frac{2n}{3}\bigr\rceil$ if $n$ is even (the two
numbers in the last minimum are equal in all cases except for the
numbers $n$ congruent to $-2$ modulo $6$ in which case the second
one is less by $1$).

Next, $\max\{na,(n+1)(n-a)/2\}=(n+1)(n-a)/2$ if $a\leq
\frac{n(n+1)}{3n+1}$ and $\max\{na,(n+1)(n-a)/2\}=na$ if $a\geq
\frac{n(n+1)}{3n+1}$. Plugging this into (\ref{d5odd}), we get
$d_{n,5}=\min\{n\bigl\lceil\frac{n(n+1)}{3n+1}\bigr\rceil,\frac{n+1}{2}
\bigl\lceil\frac{2n^2}{3n+1}\bigr\rceil\}$. Considering the cases of
$n$ being $1$, $3$ and $-1$ modulo $6$ separately, we see that
$d_{n,5}=n\bigl\lceil\frac{n(n+1)}{3n+1}\bigr\rceil$ if $n$ is
congruent to $-1$ modulo $6$ and $d_n=\frac{n+1}{2}
\bigl\lceil\frac{2n^2}{3n+1}\bigr\rceil$ ff $n\in\N$ is congruent to
$1$ or to $3$ modulo $6$.
\end{proof}

From Lemma~\ref{D5} it immediately follows that
$d_{n,5}=\frac{n^2}{3}=\phi_5 n^2$ if $6$ is a factor of $n$.
Considering the exact formula provided by Lemma~\ref{D5} and
treating the possible remainders for the division of $n$ by $6$ as
separate cases, one easily sees that $d_{n,5}-\frac{n^2}{3}\geq1$
and therefore $d_{n,5}>\lceil \phi_5n^2\rceil$ if $n$ is not
divisible by $6$ and $n\geq 3$. It is easy to verify that the
equality $d_{n,5}=\bigl\lceil\phi_5 n^2\bigr\rceil$ holds for $n=1$
and for $n=2$. This completes the Proof of Theorem~\ref{k45}.

\bigskip

We conclude by reminding that the following particular cases of the
Anick's conjecture \cite{ani1} remain unproved.

\begin{conj}\label{con1} There is a $k$-step nilpotent $\K$-algebra given by $n$
generators and $d$ quadratic relations whenever $d\geq \phi_kn^2$.
\end{conj}

\begin{conj}\label{con2} There is a finite dimensional
$\K$-algebra given by $n$ generators and $d$ quadratic relations
whenever $d>\frac{n^2}{4}$.
\end{conj}

%\vfill\break

\bigskip

%\bigskip {\bf Acknowledgements.} \ We are grateful to the referee for valuable %suggestions, which helped to improve the presentation.

{\bf Acknowledgements}

We are grateful to the Max-Planck-Institute for Mathematics in Bonn and to IHES, where  parts of this research have been done, for hospitality, support, and excellent research atmosphere.
%We would like to thank M.Kontsevich and V.Sokolov for useful conversations.
 % during various stages of the work on the paper.
This work is funded by the ERC grant 320974, and partially supported by the project PUT9038.
We also would like to thank a referee for valuable suggestions, which helped to improve the presentation.

%and the ESC Grant N9038.

\small\rm

\normalsize

\normalsize

\vskip1truecm

\scshape

\noindent  Natalia Iyudu\ \

\noindent School of Mathematics

\noindent  The University of Edinburgh

\noindent James Clerk Maxwell Building

\noindent The King's Buildings

\noindent Mayfield Road

\noindent Edinburgh

\noindent Scotland EH9 3JZ

\noindent E-mail address: \qquad {\tt niyudu@staffmail.ed.ac.uk}\ \ \

{\rm and}\ \ \

\noindent   Stanislav Shkarin

\noindent Queens's University Belfast

\noindent Pure Mathematics Research Centre

\noindent University road, Belfast, BT7 1NN, UK

\noindent E-mail address: \qquad {\tt s.shkarin@qub.ac.uk}

\end{document}